\documentclass[11pt]{article}

\usepackage[latin1]{inputenc}

\usepackage{amsfonts,amsmath,amssymb,amsthm,graphicx,epsfig,float}
\usepackage[T1]{fontenc}

\usepackage[english,francais]{babel}

{
  \newtheorem{theoreme}{Th\'eor\`eme}
  \newtheorem*{theoreme*}{Th\'eor\`eme}
  \newtheorem{lemme}[theoreme]{Lemme}
  \newtheorem{corollaire}[theoreme]{Corollaire}

  \newtheorem{proposition}[theoreme]{Proposition}
\newtheorem*{corollaire*}{Corollaire}
\newtheorem*{proposition*}{Proposition}
\theoremstyle{remark}
  \newtheorem*{remarque*}{Remarque}
}

\newcounter{ex}

\newenvironment{rem*}{
  \noindent\textbf{Remarque. }}{}

\newcommand{\Cc}{\mathbb{C}}
\newcommand{\Nn}{\mathbb{N}}

\newcommand{\Pp}{\mathbb{P}}

\title{{\bf Endomorphismes aléatoires dans les espaces projectifs I}}
\author{Henry de Thélin}
\date{}

\begin{document}
\maketitle


\def\figurename{{Fig.}}%
\def\proofname{Preuve}
\def\contentsname{Sommaire}%

\begin{abstract}

Nous étudions des suites d'endomorphismes aléatoires de $\Pp^k(\Cc)$. Sous certaines hypothèses, nous construisons un courant de Green et une mesure de Green aléatoires. Nous montrons que ces mesures de Green vérifient des propriétés de mélange.

\end{abstract}

\selectlanguage{english}
\begin{center}
{\bf{ }}
\end{center}

\begin{abstract}

We study random holomorphic endomorphisms of $\Pp^k(\Cc)$. Under some assumptions, we construct a random Green current and a random Green measure and we prove that these measures have mixing properties.

\end{abstract}

\selectlanguage{francais}

Mots-clefs: dynamique complexe, courants.

Classification: 32U40, 32H50.

\section*{{\bf Introduction}}
\par

A partir d'un endomorphisme holomorphe de $\Pp^k(\Cc)$, $f$, de degré
$d \geq 2$, J.E. Forn{\ae}ss et N. Sibony ont défini le courant de
Green $T$ associé à $f$ (voir \cite{FS} et \cite{FS1}), dont le
support est l'ensemble de Julia de $f$. Si $\omega$ désigne la forme de Fubini-Study de $\Pp^k(\Cc)$, le courant de Green est obtenu comme limite au sens des courants de la suite $\frac{(f^n)^* \omega}{d^n}$.

Ce courant possède un potentiel continu: on peut donc définir son auto-intersection $\mu= T^k$ (voir \cite{FS}). La mesure $\mu$ ainsi obtenue est l'unique
mesure d'entropie maximale $k \log(d)$ (voir \cite{BD2}) et elle a ses
exposants de Lyapunov minorés par $\frac{\log(d)}{2}$ (voir
\cite{BD1}). Par ailleurs, $\mu$ est la limite de la suite de probabilités $\frac{(f^n)^* \omega^k}{d^{kn}}$.

Les convergences des suites $\frac{(f^n)^* \omega}{d^n}$ et $\frac{(f^n)^* \omega^k}{d^{kn}}$ ont été généralisées dans plusieurs directions. L'une d'entre elle consiste à remplacer $f^n$ par $f_n \circ \dots \circ f_0$ où les $f_n$ sont soit des endomorphismes holomorphes aléatoires proches d'un endomorphisme holomorphe $f$ (voir \cite{FS2} et \cite{FW}), soit, dans le cas des mesures, des applications méromorphes qui vérifient certaines propriétés (voir \cite{DS2}). Les résultats de ce papier vont dans cette direction aussi. Précisons tout d'abord le contexte.

L'ensemble des applications rationnelles de degré $d$ de $\Pp^k(\Cc)$ forme un espace projectif $\Pp^N(\Cc)$ où $N=(k+1) \frac{(d+k)!}{d!k!} -1$. Dans cet espace $\Pp^N(\Cc)$, on notera $\mathcal{H}_d$ les points qui correspondent à des endomorphismes holomorphes de degré $d$ de $\Pp^k(\Cc)$ et $\mathcal{M}$ le complémentaire de $\mathcal{H}_d$ dans $\Pp^N(\Cc)$.

 Considérons $F$ une application mesurable de $\Pp^N(\Cc)$ dans $\Pp^N(\Cc)$ et $\Lambda$ une mesure ergodique et invariante par $F$ (par exemple $F$ un endomorphisme holomorphe de $\Pp^N(\Cc)$ et $\Lambda$ sa mesure de Green).

Si $f_0$ est un point de $\Pp^N(\Cc)$ (que l'on prendra générique pour $\Lambda$), on peut considérer la suite $f_n=F^n(f_0)$. Cela donne une suite d'applications rationnelles et on va donner des résultats sur la suite d'itérées $f_n \circ \dots \circ f_0$. L'avantage ici c'est que les $f_n$ suivent une certaine loi qui dépend de $\Lambda$. On parlera de suite d'applications rationnelles aléatoires.

Un des objectifs de cet article sera de produire dans ce contexte un courant de Green et une mesure de Green aléatoires. 

Pour $f_0 \in \Pp^N(\Cc)$ un endomorphisme holomorphe de $\Pp^k(\Cc)$, on notera $f_i= F^{i} (f_0)$ ($i \in \Nn$) et $F_n$ la composée $F_n= f_n \circ \cdots \circ f_0$ (pour $n \in \Nn$). On a 

\begin{theoreme}{\label{theoreme1}}

On suppose que 

$$\int \log dist(f, \mathcal{M}) d \Lambda(f) > - \infty.$$

Alors il existe un ensemble $A$ de mesure pleine pour $\Lambda$ tel que pour tout endomorphisme holomorphe $f_0$ de $\Pp^k(\Cc)$ avec $f_0 \in A$, on a $\frac{F_n^* \omega}{d^{n+1}}$ qui converge vers un courant $T(f_0)$. Ce courant sera appelé courant de Green aléatoire (associé à $f_0$).

\end{theoreme}

Le courant de Green aléatoire ci-dessus sera à potentiel continu: on pourra donc définir son auto-intersection $T(f_0)^l$ pour $l$ compris entre $1$ et $k$. Pour $l=k$, nous appellerons $\mu(f_0)=T(f_0)^k$ mesure de Green aléatoire (associée à $f_0$). On montrera que pour $f_0 \in A$, on a $\frac{F_n^* \omega^l}{d^{l(n+1)}} $ qui converge vers $T(f_0)^l$ quand $n$ tend vers l'infini.

L'ensemble $A$ ci-dessus vérifiera $F(A) \subset A$. En particulier, dès que $f_0$ est dans $A$ nous définirons  les courants $T(f_i)^l$ pour $i \in \Nn$. Ces courants ont des propriétés d'invariance: on a $d^{-l} f_i^* T(f_{i+1})^l=T(f_i)^l$ et $(f_i)_* T(f_i)^l= d^{k-l} T(f_{i+1})^l$. Grâce à ces invariances, en suivant des idées de \cite{DNS}, nous obtiendrons un théorème de mélange aléatoire: 

\begin{theoreme}

On considère une suite $(f_n)_{n \in \Nn}$ d'endomorphismes holomorphes de degrés $d \geq 2$ et une suite de probabilités $(\mu(f_n))_{n \in \Nn}$ telle que pour tout $n \in \Nn$ on ait $f_n^*(\mu(f_{n+1}))=d^k \mu(f_{n})$ et $\mu(f_n)=(\omega + dd^c g_n)^k$ avec $g_n$ des fonctions continues.

Alors pour $\varphi \in L^{\infty}(\Pp^k)$ et $\psi \in DSH(\Pp^k)$, on a

$$| \langle \mu(f_0), (f_{n-1} \circ \cdots \circ f_0)^* \varphi \psi \rangle - \langle \mu(f_n) , \varphi \rangle \langle \mu(f_0), \psi \rangle | \leq C d^{-n} (1+ \| g_n \|_{\infty}  )^2 \|  \varphi \|_{\infty} \| \psi \|_{\mbox{DSH}}.$$

Ici $C$ est une constante qui ne dépend que de $\Pp^k$.

\end{theoreme}

Les endomorphismes aléatoires peuvent aussi se voir d'une autre façon: en utilisant un produit semi-direct. Soit $X= \Pp^N(\Cc) \times \Pp^k(\Cc)$ et $\tau:X \longrightarrow X$ définie par $\tau(f,x)=(F(f), f(x))$. 

A partir des $\mu(f)$ construits précédemment, on peut définir une mesure $\alpha$ sur $X$ par

$$\alpha(B):= \int \mu(f)(B \cap \{f\} \times \Pp^k(\Cc)) d \Lambda(f)$$

où on identifie $\{f\} \times \Pp^k(\Cc)$ avec $\Pp^k(\Cc)$.

Nous verrons dans l'article suivant (voir \cite{Det}) que sous la même hypothèse que celle du premier théorème, la mesure $\alpha$ est bien définie et est invariante par $\tau$. Cette mesure sera ergodique et mélangeante quand $\Lambda$ le sera. Cela proviendra du théorème de mélange précédent. Nous montrerons que son entropie mixée introduite par L. M. Abramov et V. A. Rohlin (voir \cite{AR} et \cite{LW}) est maximale et vaut $k \log d$. Cette entropie mixée est une entropie métrique aléatoire pour les mesures $\mu(f)$. Nous calculerons toujours dans \cite{Det} l'entropie topologique de suites d'endomorphismes aléatoires et nous montrerons un théorème d'hyperbolicité pour les mesures $\mu(f)$ comme dans \cite{Det1}.

Voici le plan de cet article. Dans un premier paragraphe, nous faisons des rappels et nous donnons des propriétés sur les fonctions dsh et les potentiels de courants. Dans le second, nous construirons le courant et la mesure de Green aléatoires. Le troisième paragraphe sera consacré à un théorème de continuité des courants de Green aléatoires lorsque l'on fait varier l'endomorphisme de départ. Dans la dernière partie, nous donnerons le théorème de mélange aléatoire.

\section{{\bf Potentiels, fonctions dsh:}}

Dans ce paragraphe, nous allons donner des propriétés sur les potentiels et les fonctions dsh qui nous seront utiles pour démontrer les théorèmes.

\subsection{{\bf Potentiels:}}

Soit $f$ un endomorphisme holomorphe de degré $d \geq 2$ de $\Pp^k(\Cc)$ et $\omega$ la forme de Fubini-Study de $\Pp^k(\Cc)$.

L'action de $f^*$ sur $H^{1,1}(\Pp^k(\Cc))$ est la multiplication par $d$, c'est-à-dire, via le $dd^c$-lemma, qu'il existe une fonction $u_f$ telle que

$$\frac{f^* \omega}{d} = \omega + dd^c u_f.$$

Il s'agit ici de donner quelques propriétés sur cette fonction $u_f$. 

On notera $\mathcal{H}_d$ l'ensemble des endomorphismes holomorphes de degré $d$ de $\Pp^k(\Cc)$. On peut identifier $\mathcal{H}_d$ avec un ouvert de Zariski de $\Pp^N(\Cc)$ où $N=(k+1) \frac{(d+k)!}{d!k!} -1$. Le complémentaire de $\mathcal{H}_d$ dans $\Pp^N(\Cc)$ sera noté $\mathcal{M}$.

\begin{proposition}{\label{prop2}}

Il existe des constantes $C, p $ telles que pour tout $f \in \mathcal{H}_d$, on ait une fonction $u_f$ avec $\frac{f^* \omega}{d} = \omega + dd^c u_f$, $u_f \leq 0$ et

$$ \| u_f \|_{C^1} \leq C dist(f, \mathcal{M})^{-p}$$

où la distance ici est celle de Fubini-Study dans $\Pp^N(\Cc)$.

\end{proposition}

\begin{proof}

Nous utilisons ici la preuve du Théorème 2.3.1 de \cite{DS6}. Dans celle-ci, un potentiel $u_f$ est donné par:

$$u_f(z)= \int_{\xi \neq z} \frac{f^* \omega}{d}(\xi) \wedge K(z, \xi).$$

Ici $K$ est une $(k-1,k-1)$ forme négative lisse en dehors de la diagonale $\Delta$ de $\Pp^k(\Cc) \times \Pp^k(\Cc)$, dont les coefficients ont des singularités du type $\log |z - \xi | \times |z - \xi|^{2-2k}$ près de $\Delta$ et les singularités de $\nabla K$ sont en $ |z - \xi|^{1-2k}$ près de $\Delta$.

La forme $\frac{f^* \omega}{d}$ se contrôle grâce aux dérivées de $f$. On va donc estimer ces dérivées en fonction de la distance entre $f$ et l'ensemble $\mathcal{M}$.

Soit $\Phi : \Pp^k(\Cc) \times \Pp^N(\Cc) \rightarrow \Pp^k(\Cc) $ l'application $\Phi(x,f)=f(x)$. C'est une application méromorphe et son ensemble d'indétermination est inclus dans $\mathcal{A}= \Pp^k(\Cc) \times \mathcal{M}$.

En appliquant le lemme 2.1 de \cite{DiDu}, on obtient l'existence de constantes $C'$ et $p'$ telles que pour $X \notin \mathcal{A}$:

$$\| D_X \Phi \| \leq C' dist(X, \mathcal{A})^{-p'}.$$

Si $f \notin \mathcal{M}$, on a $(x,f) \notin \mathcal{A}$ pour tout $x \in \Pp^k(\Cc) $ et alors

$$ \| D_x f \| \leq \| D_{(x,f)} \Phi \| \leq C' dist((x,f), \mathcal{A})^{-p'}=C' dist(f, \mathcal{M})^{-p'} .$$

Grâce à cette majoration, on va pouvoir estimer le potentiel $u_f$.

En effet, si on se place dans des cartes de $\Pp^k(\Cc) $, on peut écrire $\omega= \sum_{i,j=1}^{k} h_{i,j} d z_i \wedge d \overline{z_j}$. La forme $f^* \omega$ est donc égale à $\sum_{i,j=1}^{k} h_{i,j}(f) d f_i \wedge d \overline{f_j}$ et on peut donc la majorer par $C'' \| D_x f \|_{\infty}^2 \omega$ où $C''$ est une constante indépendante de $f$.

La fonction $u_f$ est bien négative (car $K$ l'est) et

$$u_f(z) \geq  \int_{\xi \neq z} \frac{C''}{d}  \| D_x f \|_{\infty}^2 \omega(\xi) \wedge K(z, \xi) \geq  \int_{\xi \neq z} \frac{C''}{d}  C'^2 dist(f, \mathcal{M})^{-2p'} \omega(\xi) \wedge K(z, \xi) ,$$

d'où

$$\| u_f \|_{\infty} \leq C dist(f, \mathcal{M})^{-p}$$

avec $C$ et $p$ indépendantes de $f$.

Pour contrôler les dérivées de $u_f$ on fait comme dans le lemme 2.3.5 de \cite{DS6}: on dérive sous le signe intégral. Comme les singularités de $\nabla K$ sont en $ |z - \xi|^{1-2k}$ près de $\Delta$, on obtient que $u_f$ est $C^1$ et en faisant un raisonnement du même type que le précédent, on a 

$$ \| u_f \|_{C^1} \leq C dist(f, \mathcal{M})^{-p}$$

avec $C$ et $p$ des constantes indépendantes de $f$.

\end{proof}

A partir de la proposition précédente, on obtient:

\begin{corollaire}

Soit $K$ un compact de $\Pp^N(\Cc)$ disjoint de $\mathcal{M}$. Il existe une constante $C$ telle que pour tout $f \in K$, on ait une fonction $u_f$ avec $\frac{f^* \omega}{d} = \omega + dd^c u_f$, $u_f \leq 0$ et

$$ \| u_f \|_{C^1} \leq C .$$

\end{corollaire}

Dans la suite, on aura besoin de comparer deux potentiels $u_f$ et $u_g$ pour $f$ et $g$ dans $\mathcal{H}_d$. Pour cela, on a 
 
 \begin{proposition}{\label{prop4}}
 
 Soit $K$ un compact de $\Pp^N(\Cc)$ disjoint de $\mathcal{M}$.  Il existe une constante $C$ telle que pour tout $f,g \in K$ les potentiels $u_f$ et $u_g$ de la proposition précédente vérifient:
 
 $$ \| u_f -u_g \|_{\infty} \leq C dist(f,g).$$
 
 \end{proposition} 

Pour démontrer cette proposition, on utilisera le lemme suivant:

\begin{lemme}

Soit $K$ un compact de $\Pp^N(\Cc)$ disjoint de $\mathcal{M}$. Il existe une constante $C$ telle que pour tout $f, g \in K$ on ait

$$-C dist(f,g) \omega \leq \frac{f^* \omega}{d}- \frac{g^* \omega}{d} \leq C dist(f,g) \omega.$$

\end{lemme}

\begin{proof}

Quitte à considérer une partition de l'unité, on peut supposer que $\omega$ est à support dans un compact $L$ d'une carte. Il suffit aussi de comparer $\frac{f^* \omega}{d}$ et $\frac{g^* \omega}{d}$ sur un ouvert $U$ relativement compact dans une carte de $\Pp^k(\Cc)$.

Si on écrit $\omega= \sum_{i,j=1}^{k} h_{i,j} d z_i \wedge d \overline{z_j}$, on a $\frac{f^* \omega}{d}$ qui est égale à $\sum_{i,j=1}^{k} \frac{h_{i,j}(f)}{d} d f_i \wedge d \overline{f_j}$. La différence entre $\frac{f^* \omega}{d}$ et $\frac{g^* \omega}{d}$ est donc une somme finie de termes (dont le nombre ne dépend que de la dimension $k$) du type:  

$$ \left( \frac{h_{i,j}(f)}{d} \frac{\partial f_i}{\partial z_l} \frac{\partial \overline{f_j}}{\partial \overline{z_p}} - \frac{h_{i,j}(g)}{d} \frac{\partial g_i}{\partial z_l} \frac{\partial \overline{g_j}}{\partial \overline{z_p}} \right) dz_l \wedge d \overline{ z_p}.$$

En module, le coefficient de ce terme est plus petit que

\begin{equation*}
\begin{split}
 I=&\left|  \frac{h_{i,j}(f)}{d} \frac{\partial f_i}{\partial z_l} \frac{\partial \overline{f_j}}{\partial \overline{z_p}} - \frac{h_{i,j}(f)}{d} \frac{\partial g_i}{\partial z_l} \frac{\partial \overline{g_j}}{\partial \overline{z_p}} \right| + \left| \frac{h_{i,j}(f)}{d} \frac{\partial g_i}{\partial z_l} \frac{\partial \overline{g_j}}{\partial \overline{z_p}} - \frac{h_{i,j}(g)}{d} \frac{\partial g_i}{\partial z_l} \frac{\partial \overline{g_j}}{\partial \overline{z_p}} \right|\\
& \leq \left| \frac{h_{i,j}}{d} \right|_{\infty, L} \left| \frac{\partial f_i}{\partial z_l} \frac{\partial \overline{f_j}}{\partial \overline{z_p}} -  \frac{\partial g_i}{\partial z_l} \frac{\partial \overline{g_j}}{\partial \overline{z_p}} \right| + \| D_zg \|^2 \left| \frac{h_{i,j}(f)}{d} - \frac{h_{i,j}(g)}{d} \right|.
\end{split}
\end{equation*}

Ici $| \cdot |_{\infty, L}$ désigne le sup du module de la fonction sur $L$.

En utilisant l'inégalité $|ab-a'b'| \leq |a| |b-b'| + |b'| |a-a'|$ on obtient que (si $z$ est le point où on estime le coefficient)

$$I \leq C'  (\| D_z f\|\| D_z f - D_z g \|+  \| D_z g\|\| D_z f - D_z g \|  + \| D_zg \|^2 dist(f,g)).$$

Comme $f$ et $g$ sont dans le compact $K$ qui est disjoint de $\mathcal{M}$, il existe une constante $C''$ avec $\| D_z f\|_{\infty} \leq C''$ et $\| D_z g\|_{\infty} \leq C''$ (voir la preuve de la proposition précédente).

Maintenant, si on reprend la fonction $\Phi : \Pp^k(\Cc) \times \Pp^N(\Cc) \rightarrow \Pp^k(\Cc) $ définie par $\Phi(x,f)=f(x)$, on a (toujours par le lemme 2.1 de \cite{DiDu}) l'existence de constantes $C'''$ et $p'''$ telles que pour $X \notin \mathcal{A}$:

$$\| D_X \Phi \| + \|D_X^2 \Phi \| \leq C''' dist(X, \mathcal{A})^{-p'''}.$$

On en déduit que

$$\| D_z f - D_z g \| \leq \| D_{(z ,f)} \Phi - D_{(z ,g)} \Phi \| \leq C''' dist([(z,f),(z,g)], \mathcal{A})^{-p'''} dist((z,f),(z,g))$$

où $[(z,f),(z,g)]$ désigne un segment vertical qui rejoint $(z,f)$ à $(z,g)$.

Soit $\epsilon_0 >0$ plus petit que la distance de $K$ à $\mathcal{M}$ divisée par $2$.

Si $g$ est dans la boule de centre $f$ et de rayon $\epsilon_0$, on obtient que

$$dist([(z,f),(z,g)], \mathcal{A})^{-p'''} \leq \epsilon_0^{-p'''}$$

indépendamment de $z$, d'où

$$\| D_z f - D_z g \|_{\infty} \leq C'''  \epsilon_0^{-p'''} dist(f,g).$$

On obtient alors que

$$I \leq C dist(f,g) $$

où $C$ est indépendante de $f$ et $g$ pour $g$ dans la boule de centre $f$ et de rayon $\epsilon_0$. Si $g$ est en dehors de cette boule, quitte à prendre $C$ plus grande, cette inégalité est toujours vraie car $I$ est uniformément majorée pour $f, g \in K$.

Cela démontre le lemme.

\end{proof}

Grâce à ce lemme, on obtient immédiatement la proposition précédente car par la preuve de la proposition \ref{prop2}, on a

$$u_f(z)-u_g(z)= \int_{\xi \neq z} \left( \frac{f^* \omega}{d}(\xi) - \frac{g^* \omega}{d}(\xi) \right) \wedge K(z, \xi)$$

et $K$ est négative.

\subsection{\bf{Fonctions dsh:}}

Commençons par quelques rappels sur les fonctions dsh (voir par exemple \cite{DS2}).

Une fonction $\varphi$ est {\it quasi-plurisousharmonique} (qpsh) si elle s'écrit localement comme la somme d'une fonction psh et d'une fonction $C^{\infty}$ . Une telle fonction vérifie $dd^c \varphi \geq -c \omega$ au sens des courants pour une constante $c \geq 0$. 

Un ensemble de $\Pp^k(\Cc)$ est dit {\it pluripolaire} s'il est inclus dans $\{ \varphi = - \infty \}$ où $\varphi$ est une fonction qpsh. On appelle maintenant {\it fonction dsh} toute fonction, définie en dehors d'un pluripolaire, qui s'écrit comme différence de deux fonctions qpsh. Notons $\mbox{DSH}(\Pp^k(\Cc))$ l'ensemble des fonctions dsh sur $\Pp^k(\Cc)$. Si $\varphi$  est une fonction dsh, il existe deux courants $T^{\pm}$ positifs fermés de bidegré $(1,1)$ tels que $dd^c \varphi= T^+ - T^-$. On peut alors définir une norme (voir \cite{DNS} paragraphe 3):

$$\| \varphi \|_{\mbox{DSH}} := \| \varphi \|_{L^1(\Pp^k(\Cc))} + \inf \| T^{\pm} \|$$

avec $T^{\pm}$ comme précédemment.

Soit $\mu$ une mesure qui s'écrit $\mu=T^k$ où $T$ est un $(1,1)$ courant positif à potentiel continu. On peut alors définir une norme équivalente à la précédente (voir \cite{DNS}):

$$\| \varphi \|^{\mu}_{\mbox{DSH}}:= | \langle \mu, \varphi \rangle | + \inf \| T^{\pm} \|.$$

Dans cet article, nous aurons besoin d'estimations sur les constantes en jeu pour l'équivalence des normes. Ce sera l'objet des deux propositions suivantes.

\begin{proposition}{\label{estimation1}}

Soit $\mu$ une mesure qui s'écrit $\mu=T^k$ où $T$ est un courant positif de la forme $T= \omega + dd^c g$ avec $g$ fonction continue. Alors pour toute fonction $\varphi$ DSH, on a

$$ | \langle \mu, \varphi \rangle |  \leq C(1+ \|g\|_{\infty} ) \| \varphi \|_{\mbox{DSH}}$$

avec $C$ qui ne dépend que de $\Pp^k(\Cc)$. En particulier,  

$$ \| \varphi \|^{\mu}_{\mbox{DSH}}  \leq C'(1+ \|g\|_{\infty} ) \| \varphi \|_{\mbox{DSH}}$$

avec $C'$ qui ne dépend que de $\Pp^k(\Cc)$.

\end{proposition}

\begin{proof}

On a 

\begin{equation*}
\begin{split}
| \langle \mu, \varphi \rangle |=| \langle T^k, \varphi \rangle | & \leq | \langle \omega \wedge T^{k-1}, \varphi \rangle | + | \langle dd^c g \wedge T^{k-1} , \varphi \rangle |\\
& =| \langle \omega \wedge T^{k-1}, \varphi \rangle | + | \langle  g ,  T^{k-1}  \wedge dd^c \varphi \rangle| \\
& \leq | \langle \omega \wedge T^{k-1}, \varphi \rangle | + | \langle  g ,  T^{k-1}  \wedge T^+ \rangle|  + | \langle  g ,  T^{k-1}  \wedge T^- \rangle|
\end{split}
\end{equation*}

où $dd^c \varphi= T^+ - T^-$ comme précédemment. Mais $T^{k-1}  \wedge T^{\pm}$ sont des mesures positives de masse $\| T^{+} \|$. On a donc

$$| \langle \mu, \varphi \rangle | \leq | \langle \omega \wedge T^{k-1}, \varphi \rangle | + 2 \|g\|_{\infty} \| \varphi \|_{\mbox{DSH}}.$$

Maintenant, il s'agit de recommencer ce que l'on vient de faire avec $\langle \omega \wedge T^{k-1}, \varphi \rangle |$ à la place de $| \langle T^k, \varphi \rangle | $ et ainsi de suite.

 A la fin, on obtient

$$| \langle \mu, \varphi \rangle | \leq | \langle \omega^k , \varphi \rangle | + 2k \|g\|_{\infty} \| \varphi \|_{\mbox{DSH}},$$

ce qui donne la proposition.

\end{proof}

On va maintenant donner une estimée pour la constante de l'autre inégalité d'équivalence de norme.

\begin{proposition}{\label{estimation2}}

Soit $\mu$ une mesure qui s'écrit $\mu=T^k$ où $T$ est un courant positif de la forme $T= \omega + dd^c g$ avec $g$ fonction continue. Alors pour toute fonction $\varphi$ DSH, on a

$$\| \varphi \|_{\mbox{DSH}} \leq C (1+ \|g\|_{\infty} ) \| \varphi \|^{\mu}_{\mbox{DSH}}$$

où $C$ est une constante qui ne dépend que de $\Pp^k(\Cc)$.

\end{proposition}

\begin{proof}

Il s'agit ici de montrer

$$\| \varphi \|_{L^1(\Pp^k(\Cc))}  \leq C' (1+ \|g\|_{\infty} ) \| \varphi \|^{\mu}_{\mbox{DSH}}$$

où $C'$ ne dépend que de $\Pp^k(\Cc)$.

La fonction $\varphi $ est DSH donc on peut écrire $dd^c \varphi =T^+ - T^-$ avec $T^{\pm}$ des courants positifs fermés. En utilisant le théorème 2.3.1 de \cite{DS6}, on obtient l'existence de fonctions qpsh $\varphi^{\pm}$ telles que

$$dd^c \varphi^+ - dd^c  \varphi^-=T^+ - T^- \mbox{  et  }  \| \varphi^{\pm} \|_{\mbox{DSH}} \leq A \| T^{\pm} \|$$

où $A$ est une constante qui ne dépend que de $\Pp^k(\Cc)$.

Comme $\varphi$ est DSH, elle s'écrit en dehors d'un pluripolaire, $\varphi= \varphi_1 - \varphi_2$ où $\varphi_1$ et $ \varphi_2$ sont des fonctions qpsh. On a donc, presque partout, $\varphi_1 - \varphi_2=     \varphi^+ -  \varphi^- + a$ où $a$ est une constante. Comme les fonctions sont qpsh, on a l'égalité partout et alors $\varphi=\varphi^+ -  \varphi^- + a$ en dehors d'un pluripolaire.

Maintenant, on a

\begin{equation*}
\begin{split}
\| \varphi \|_{L^1(\Pp^k(\Cc))}  & \leq  \| \varphi^+ \|_{L^1(\Pp^k(\Cc))}  +   \| \varphi^- \|_{L^1(\Pp^k(\Cc))} + |a|\\
& \leq 2A \|T^{\pm} \| + |a| \leq 2A \| \varphi \|^{\mu}_{\mbox{DSH}} + |a|.
\end{split}
\end{equation*}

Il s'agit donc d'estimer la constante $a$.

Comme $\mu=T^k$ et que le potentiel de $T$ est continu, $\mu$ ne charge pas les ensembles pluripolaires et on a:

$$ \langle \mu , \varphi \rangle= \langle \mu , \varphi^+ -  \varphi^-  \rangle + \langle \mu , a \rangle= \langle \mu , \varphi^+ -  \varphi^-  \rangle + a.$$

Ce qui implique, par la proposition précédente,

\begin{equation*}
\begin{split}
|a| &\leq  | \langle  \mu , \varphi^+ \rangle | + | \langle  \mu , \varphi^- \rangle | + |\langle \mu , \varphi \rangle |\\
& \leq \| \varphi^+ \|^{\mu}_{\mbox{DSH}} + \| \varphi^- \|^{\mu}_{\mbox{DSH}} + \| \varphi \|^{\mu}_{\mbox{DSH}}\\
& \leq C'(1+ \|g\|_{\infty} ) \| \varphi^+ \|_{\mbox{DSH}} + C'(1+ \|g\|_{\infty} ) \| \varphi^- \|_{\mbox{DSH}} +  \| \varphi \|^{\mu}_{\mbox{DSH}}\\
& \leq 2 A C'(1+ \|g\|_{\infty} ) \| T^{\pm} \| + \| \varphi \|^{\mu}_{\mbox{DSH}}\\
& \leq 2 A C'(1+ \|g\|_{\infty} ) \| \varphi \|^{\mu}_{\mbox{DSH}} + \| \varphi \|^{\mu}_{\mbox{DSH}}.
\end{split}
\end{equation*}

Cela démontre la proposition.

\end{proof}

\section{\bf{Construction du courant de Green et de la mesure de Green aléatoires}}

Dans ce paragraphe, nous construisons le courant de Green et la mesure de Green aléatoires. Par rapport à \cite{FS2} et \cite{FW}, nos endomorphismes aléatoires $f_1, \cdots , f_n$ ne restent pas forcément autour d'un endomorphisme holomorphe, il peuvent s'approcher des applications méromorphes. Cependant, la condition d'intégrabilité de $\log dist(f, \mathcal{M}) $ et le fait que les endomorphismes suivent en quelque sorte une loi dictée par $\Lambda$, va faire que l'approche des applications méromorphes sera bonne.

\subsection{\bf{Construction du courant de Green}}{\label{courant}}

Commençons par démontrer le théorème \ref{theoreme1}. 

Précisons tout d'abord l'ensemble $A$. Comme on a $\int \log dist(f, \mathcal{M}) d \Lambda(f) > - \infty$, le théorème de Birkhoff implique qu'il existe un ensemble de mesure pleine pour $\Lambda$ (que l'on appellera $A$) tel que pour $f_0 \in A$ on ait

$$\frac{1}{n} \sum_{i=0}^{n-1} \log dist(F^{i}(f_0), \mathcal{M}) \rightarrow \int \log dist(f, \mathcal{M}) d \Lambda(f) > - \infty.$$

Fixons $f_0$ dans $ A$. Il s'agit de construire le courant de Green $T(f_0)$ associé.

Tout d'abord, grâce à la convergence précédente, on a

\begin{lemme}{\label{lemme}}
Pour tout $\epsilon>0$, il existe $n_0 \in \Nn$ tel que $dist(F^{n}(f_0), \mathcal{M}) \geq e^{- \epsilon n}$ pour $n \geq n_0$.

\end{lemme}

Admettons pour l'instant ce lemme et utilisons le avec $\epsilon$ petit devant $\frac{\log d}{p}$ où le $p$ est défini dans la proposition \ref{prop2}.

Par cette proposition, il existe des fonctions $u_i$ telles que $\frac{f_i^* \omega}{d} = \omega + dd^c u_i$ (où ici $f_i=F^{i}(f_0)$) et

$$\| u_i \|_{\infty} \leq C dist(f_i , \mathcal{M})^{-p} = C dist(F^{i}(f_0) , \mathcal{M})^{-p} \leq C e^{i \epsilon p}$$

pour $i \geq n_0$.

Cette estimée va montrer que $\frac{F_n^* \omega}{d^{n+1}}$ converge vers un courant $T(f_0)$ (ici on note $F_n = f_n \circ \cdots \circ f_0$).

Pour cela, on va utiliser la méthode classique de construction de courant de Green via des arguments de cohomologie.

On a tout d'abord $\frac{f_n^* \omega}{d} = \omega + dd^c u_n$. Ensuite, si on applique $\frac{f_{n-1}^*}{d}$ on obtient

$$\frac{f_{n-1}^*}{d}  \left( \frac{f_n^* \omega}{d} \right)=\frac{f_{n-1}^* \omega}{d} + \frac{dd^c u_n \circ f_{n-1}}{d} = \omega + dd^c u_{n-1} + \frac{dd^c u_n \circ f_{n-1}}{d} $$

et en recommençant ce procédé, on a:

$$\frac{(f_n \circ \cdots \circ f_0)^* \omega}{d^{n+1}}=\frac{f_0^* \circ \cdots \circ f_n^* \omega}{d^{n+1}}=\omega + dd^c \left( \sum_{i=0}^{n} \frac{u_i \circ f_{i-1} \circ \cdots \circ f_0}{d^{i}} \right).$$

Notons 

$$g_n=\sum_{i=0}^{n} \frac{u_i \circ f_{i-1} \circ \cdots \circ f_0}{d^{i}}.$$

L'hypothèse de convergence

$$\frac{1}{n} \sum_{i=0}^{n-1} \log dist(F^{i}(f_0), \mathcal{M}) \rightarrow \int \log dist(f, \mathcal{M}) d \Lambda(f) > - \infty$$

implique que les $f_i$ sont des endomorphismes holomorphes et que les $u_i$ sont continues (et même lisses).

En particulier $(g_n)$ est une suite de fonctions continues et on va montrer qu'elle vérifie le critère de Cauchy uniforme. Si $m \geq n \geq n_0$, on a

\begin{equation*}
\begin{split}
\|g_m - g_n \|_{\infty} &= \left\| \sum_{i=n+1}^m \frac{u_i \circ f_{i-1} \circ \cdots \circ f_0}{d^{i}} \right\|_{\infty} \leq \sum_{i=n+1}^m \frac{\| u_i \|_{\infty}}{d^i}\\
&\leq  C \sum_{i=n+1}^m  \frac{e^{i \epsilon p}}{d^i}= C \sum_{i=n+1}^m  \left( \frac{e^{\epsilon p}}{d} \right)^i  \leq C \sum_{i=n+1}^{+ \infty}  \left( \frac{e^{\epsilon p}}{d} \right)^i 
 \end{split}
\end{equation*}

et le denier terme est aussi petit que l'on veut si $n$ est assez grand.

La suite $(g_n)$ converge donc uniformément vers une fonction $g$ continue.

Cela implique que $\frac{F_n^* \omega}{d^{n+1}}= \omega + dd^c g_n$ converge vers $\omega + dd^c g$ au sens des courants. C'est ce dernier courant que l'on appellera courant de Green aléatoire de $f_0$ et que l'on notera $T(f_0)$.

Il reste maintenant à démontrer le lemme précédent.

\medskip

{\bf{Démonstration du lemme \ref{lemme}:}}

Soit $\epsilon >0$. Par hypothèse, on a 

$$\frac{1}{n} \sum_{i=0}^{n-1} \log dist(F^{i}(f_0), \mathcal{M}) \rightarrow \int \log dist(f, \mathcal{M}) d \Lambda(f) = C_0 > - \infty.$$

Il existe donc $n_0 \in \Nn$ tel que pour $n \geq n_0$, on ait

$$\left| \frac{1}{n} \sum_{i=0}^{n-1} \log dist(F^{i}(f_0), \mathcal{M})-C_0 \right| < \frac{\epsilon}{3}$$

et ainsi

$$\left| \sum_{i=0}^{n} \log dist(F^{i}(f_0), \mathcal{M}) -(n+1)C_0 - \sum_{i=0}^{n-1} \log dist(F^{i}(f_0), \mathcal{M}) + n C_0 \right| < (2n+1) \frac{\epsilon}{3}.$$

Pour $n \geq n_0$, on a alors

$$ | \log dist(F^{n}(f_0), \mathcal{M}) | \leq (2n+1) \frac{\epsilon}{3} + C_0$$

et le terme de droite est plus petit que $\epsilon n$ si $n$ est suffisamment grand.

Cela implique le résultat.

Remarquons aussi que l'on a une certaine invariance de l'ensemble $A$:

\begin{lemme}

On a $F(A) \subset A$.

\end{lemme}

\begin{proof}

Rappelons que $A$ est l'ensemble des points $f_0$ tels que

$$\frac{1}{n} \sum_{i=0}^{n-1} \log dist(F^{i}(f_0), \mathcal{M}) \rightarrow \int \log dist(f, \mathcal{M}) d \Lambda(f) > - \infty.$$

Soit $f_0 \in A$. Dans la démonstration du lemme précédent, on a montré que $\frac{1}{n} \log dist(F^{n}(f_0), \mathcal{M}) $ converge vers $0$. On a donc

\begin{equation*}
\begin{split}
&\frac{1}{n} \sum_{i=0}^{n-1} \log dist(F^{i}(f_1), \mathcal{M}) =\frac{1}{n} \sum_{i=0}^{n-1} \log dist(F^{i+1}(f_0), \mathcal{M})\\
&=\frac{1}{n} \sum_{i=0}^{n-1} \log dist(F^{i}(f_0), \mathcal{M}) - \frac{1}{n} \log dist(f_0, \mathcal{M}) + \frac{1}{n} \log dist( F^n(f_0), \mathcal{M})
\end{split}
\end{equation*}

qui converge bien vers $\int \log dist(f, \mathcal{M}) d \Lambda(f) > - \infty$ car $f_0 \in A$.

\end{proof}

\subsection{\bf{Courants d'ordre supérieur}}

Le courant $T(f_0)=\omega + dd^c g$ construit dans le paragraphe précédent est à potentiel continu. On peut donc considérer $T(f_0)^l$ et en particulier $\mu(f_0)=T(f_0)^k$ que l'on appellera mesure de Green aléatoire associée à $f_0$. On a aussi un résultat de convergence pour ces courants:

\begin{proposition}{\label{prop9}}

Si $f_0 \in A$, on a $\frac{F_n^* \omega^l}{d^{l(n+1)}} $ qui converge vers $T(f_0)^l$ quand $n$ tend vers l'infini.

\end{proposition}

\begin{proof}

On va faire une récurrence sur $l$.

Pour $l=1$ c'est le résultat du paragraphe précédent. Supposons la propriété vraie au rang $l-1$ et montrons que $\frac{F_n^* \omega^l}{d^{l(n+1)}} $ converge vers $T(f_0)^l$. On a

$$\frac{F_n^* \omega^l}{d^{l(n+1)}}= \frac{F_n^* \omega^{l-1}}{d^{(l-1)(n+1)}} \wedge  \frac{F_n^* \omega}{d^{n+1}}=\frac{F_n^* \omega^{l-1}}{d^{(l-1)(n+1)}} \wedge T(f_0) + \frac{F_n^* \omega^{l-1}}{d^{(l-1)(n+1)}} \wedge \left( \frac{F_n^* \omega}{d^{n+1}} -T(f_0) \right). $$

Commençons par le premier terme $\frac{F_n^* \omega^{l-1}}{d^{(l-1)(n+1)}} \wedge T(f_0) $. Il est égal à 

$$\frac{F_n^* \omega^{l-1}}{d^{(l-1)(n+1)}} \wedge \omega + \frac{F_n^* \omega^{l-1}}{d^{(l-1)(n+1)}} \wedge dd^c g .$$

Mais par hypothèse de récurrence, $\frac{F_n^* \omega^{l-1}}{d^{(l-1)(n+1)}} \wedge \omega$ converge vers $T(f_0)^{l-1} \wedge \omega$ et si $\psi$ est une forme test on a

$$ \left\langle \frac{F_n^* \omega^{l-1}}{d^{(l-1)(n+1)}} \wedge dd^c g , \psi \right\rangle= \left\langle \frac{F_n^* \omega^{l-1}}{d^{(l-1)(n+1)}} \wedge dd^c \psi  , g \right\rangle$$

qui converge vers

$$\langle T(f_0)^{l-1} \wedge dd^c \psi, g \rangle=\langle T(f_0)^{l-1} \wedge dd^c g, \psi \rangle$$

car $g$ est continue.

En résumé, le premier terme $\frac{F_n^* \omega^{l-1}}{d^{(l-1)(n+1)}} \wedge T(f_0) $ tend vers $T(f_0)^l$ quand $n$ tend vers l'infini.

Pour finir la démonstration, il reste à montrer que la deuxième partie $\frac{F_n^* \omega^{l-1}}{d^{(l-1)(n+1)}} \wedge \left( \frac{F_n^* \omega}{d^{n+1}} -T(f_0) \right)$ converge vers $0$.

Soit $\psi$ une $(k-l,k-l)$ forme test. On a

\begin{equation*}
\begin{split}
&\left| \left\langle \frac{F_n^* \omega^{l-1}}{d^{(l-1)(n+1)}} \wedge \left( \frac{F_n^* \omega}{d^{n+1}} -T(f_0) \right) , \psi \right\rangle \right|=\left| \left\langle \frac{F_n^* \omega^{l-1}}{d^{(l-1)(n+1)}} \wedge \left( -dd^c \sum_{i=n+1}^{+ \infty} \frac{u_i \circ f_{i-1} \circ \cdots \circ f_0}{d^i} \right) , \psi \right\rangle \right|\\
&=\left| \left\langle \frac{F_n^* \omega^{l-1}}{d^{(l-1)(n+1)}} \wedge dd^c \psi, - \sum_{i=n+1}^{+ \infty} \frac{u_i \circ f_{i-1} \circ \cdots \circ f_0}{d^i} \right\rangle \right|. \\
\end{split}
\end{equation*}

Maintenant, on peut trouver une constante $C(\psi)$ telle que $-C(\psi) \omega^{k-l+1} \leq dd^c \psi \leq C(\psi) \omega^{k-l+1} $. En écrivant $dd^c \psi= (dd^c \psi + C(\psi) \omega^{k-l+1})-C(\psi) \omega^{k-l+1}$ dans le dernier terme, on obtient

$$\left| \left\langle \frac{F_n^* \omega^{l-1}}{d^{(l-1)(n+1)}} \wedge \left( \frac{F_n^* \omega}{d^{n+1}} -T(f_0) \right) , \psi \right\rangle \right| \leq 3C(\psi) \left\langle \frac{F_n^* \omega^{l-1}}{d^{(l-1)(n+1)}} \wedge \omega^{k-l+1}, - \sum_{i=n+1}^{+ \infty} \frac{u_i \circ f_{i-1} \circ \cdots \circ f_0}{d^i} \right\rangle .$$

car les $u_i$ sont négatives.

Enfin, ce dernier terme est plus petit que

$$3C(\psi) \left\| \frac{F_n^* \omega^{l-1}}{d^{(l-1)(n+1)}} \right\| \sum_{i=n+1}^{+ \infty} \frac{\| u_i \|_{\infty} }{d^i} \leq 3C(\psi) \sum_{i=n+1}^{+ \infty}  \frac{e^{i \epsilon p}} {d^i}   $$

pour $0 < \epsilon < \frac{\log d}{p}$ fixé et $n \geq n_0$ avec le $n_0$ du lemme précédent.

Cela implique bien que $\frac{F_n^* \omega^{l-1}}{d^{(l-1)(n+1)}} \wedge \left( \frac{F_n^* \omega}{d^{n+1}} -T(f_0) \right)$ converge vers $0$ au sens des courants quand $n$ tend vers l'infini.

\end{proof}

\subsection{\bf{Propriétés d'invariance}}

Dans ce paragraphe, nous montrons que les courants de Green aléatoires possèdent des propriétés d'invariance. 

On a montré que pour $f_0 \in A$, on pouvait construire un courant de Green $T(f_0)$. Comme $F(A) \subset A$, si $f_0 \in A$ on a l'existence de courants de Green $T(f_i)$ (avec $i \geq 0$) où $f_i=F^i(f_0)$.

Ces courants vérifient

\begin{proposition}

Soit $f_0 \in A$. Alors, 

$$\frac{f_i^* T(f_{i+1})^l}{d^l}=T(f_i)^l.$$

En particulier, on a  

$$\frac{f_i^* \mu (f_{i+1})}{d^k}=\mu(f_i).$$

\end{proposition}

\begin{proof}

Si on note $F_{i+1,n}=f_{i+1+n} \circ \cdots \circ f_{i+1}$, on a prouvé que $\frac{F_{i+1,n}^* \omega^l}{d^{l(n+1)}}$ converge vers $T(f_{i+1})^l$ quand $n$ tend vers l'infini.

Comme tous les $f_i$ sont dans $A$, ils sont en particulier holomorphes et en utilisant \cite{DS3}, on a d'une part $f_i^* \left( \frac{F_{i+1,n}^* \omega^l}{d^{l(n+1)}} \right) $ qui tend vers $f_i^*(T(f_{i+1})^l)$ au sens des courants quand $n$ tend vers l'infini, d'autre part,

$$f_i^* \left( \frac{F_{i+1,n}^* \omega^l}{d^{l(n+1)}} \right)= \frac{F_{i,n+1}^* \omega^l}{d^{l(n+1+1)}} \times \frac{d^{l(n+1+1)}}{d^{l(n+1)}}=d^l \frac{F_{i,n+1}^* \omega^l}{d^{l(n+1+1)}}$$

converge vers $d^l T(f_i)^l$. Cela prouve la proposition.

\end{proof}

Comme on a $f_* f^* = d^k Id$ sur l'ensemble des courants positifs fermés, on a aussi

\begin{corollaire}  

Pour $l$ compris entre $1$ et $k$, $(f_i)_* T(f_i)^l= d^{k-l} T(f_{i+1})^l$.

En particulier $(f_i)_* \mu(f_i)=  \mu(f_{i+1})$.

\end{corollaire}

\section{\bf{Un théorème de convergence}}{\label{convergence}}

Dans ce paragraphe, nous allons perturber l'endomorphisme holomorphe de départ $f_0 \in A$. On va noter pour cela $f_{0,0}=f_0$ et on va considérer une suite $(f_{n,0})_{n \geq 1}$ de $A$ qui converge vers $f_{0,0}$ quand $n$ tend vers l'infini. Nous montrons le théorème suivant (le $p$ est celui de la proposition \ref{prop2})

\begin{theoreme}

Supposons que $F$ soit continue et que

$$\mbox{(H)} \mbox{  } \mbox{  }\mbox{  }\mbox{  } \sum_{i=0}^{+ \infty} \sup_{n \geq 0} \frac{dist(F^i(f_{n,0}), \mathcal{M})^{-p}}{d^i} < + \infty.$$

Alors, $T(f_{n,0})^l$ converge vers $T(f_0)^l$ quand $n$ tend vers l'infini.

\end{theoreme}

 Nous expliciterons ensuite une situation où l'on peut appliquer ce théorème.

Commençons par le cas $l=1$.

On pose $f_{n,i}= F^i (f_{n,0})$ pour $i,n \in \Nn$. L'hypothèse (H) et la proposition \ref{prop2} impliquent l'existence de fonctions $u_{n,i}$ avec

$$\frac{f_{n,i}^*}{d}= \omega + dd^c u_{n,i}$$

et 

$$ \sum_{i=0}^{+ \infty} \sup_{n \geq 0} \frac{ \| u_{n,i} \|_{\infty}}{d^i} < + \infty.$$

Par construction, nous avons

$$T(f_{n,0})= \omega + dd^c \left( \sum_{i=0}^{\infty} \frac{u_{n,i} \circ f_{n,i-1} \circ \cdots \circ f_{n,0}}{d^{i}} \right).$$
 
Si on note

$$g_n=\sum_{i=0}^{\infty} \frac{u_{n,i} \circ f_{n,i-1} \circ \cdots \circ f_{n,0}}{d^{i}} ,$$

on va montrer que $(g_n)_{n \geq 1}$ converge uniformément vers $g_0$.

Soit $\epsilon > 0$. Pour $m \geq 0$, on a 

$$\left\| \sum_{i=m}^{\infty} \frac{u_{n,i} \circ f_{n,i-1} \circ \cdots \circ f_{n,0}}{d^{i}} \right\|_{\infty} \leq \sum_{i=m}^{\infty} \frac{\|u_{n,i} \|_{\infty} }{d^{i}} \leq \sum_{i=m}^{\infty} \sup_{n \geq 0} \frac{\|u_{n,i} \|_{\infty} }{d^{i}}$$

et le terme de droite tend vers $0$ (indépendamment de $n$) quand $m$ tend vers l'infini par l'hypothèse (H). Soit donc $m_0$ tel que $\sum_{i=m_0}^{\infty} \sup_{n \geq 0} \frac{\|u_{n,i} \|_{\infty} }{d^{i}} \leq \epsilon$.

On a 

$$\|g_n - g_0 \|_{\infty} \leq \left\| \sum_{i=0}^{m_0-1} \frac{u_{n,i} \circ f_{n,i-1} \circ \cdots \circ f_{n,0}-u_{0,i} \circ f_{0,i-1} \circ \cdots \circ f_{0,0}}{d^{i}} \right\|_{\infty} + 2 \epsilon$$

pour tout $n \geq 0$.

Mais, l'hypothèse (H) implique que les $F^i(f_{0,0})=f_{0,i}$ sont holomorphes. Comme les $F^i$ sont continues au point $f_{0,0}$, on a $F^i(f_{n,0})=f_{n,i}$ qui converge vers $F^i(f_{0,0})=f_{0,i}$.

En utilisant le lemme suivant (qui découle juste du théorème des accroissements finis)

\begin{lemme}

Il existe une constante $C$ telle que pour tous $f_n , g_n , f, g \in \mathcal{H}_d$ on ait

$$dist(f_n \circ g_n , f \circ g) \leq dist(f_n, f) + C \| D_x f \|_{\infty} dist(g_n,g).$$

Ici, la distance entre deux fonctions $f$ et $g$ est donnée par $dist(f,g)= \max_{\Pp^k(\Cc)} dist(f(x),g(x))$.

\end{lemme}

on obtient par récurrence que

$$\max_{x \in \Pp^k(\Cc)} dist(f_{n,i-1} \circ \cdots \circ f_{n,0}(x), f_{0,i-1} \circ \cdots \circ f_{0,0}(x))$$

converge vers $0$ quand $n$ tend vers l'infini pour $i=0 , \cdots , m_0 - 1$.

Comme $f_{n,i}$ tend vers $f_{0,i}$ quand $n$ tend vers l'infini, la proposition \ref{prop4} implique que  pour $i=0 , \cdots , m_0 - 1$, on a  $\| u_{n,i} - u_{0,i} \|_{\infty} $ qui converge vers $0$ et en utilisant un analogue au lemme précédent (les fonctions $u_{0,i}$ sont $C^1$ par la proposition \ref{prop2}), on obtient que 

$$\left\| \sum_{i=0}^{m_0-1} \frac{u_{n,i} \circ f_{n,i-1} \circ \cdots \circ f_{n,0}-u_{0,i} \circ f_{0,i-1} \circ \cdots \circ f_{0,0}}{d^{i}} \right\|_{\infty}$$

tend vers $0$ quand $n$ tend vers l'infini.

La suite $(g_n)_{n \geq 1}$ converge donc uniformément vers $g_0$.

Cette convergence uniforme implique que $T(f_{n,0})= \omega + dd^c g_n$ converge au sens des courants vers $T(f_{0,0})= \omega + dd^c g_0$.

Pour $l\geq 1$, on fait une récurrence. Le cas $l=1$ vient d'être démontré. Supposons la propriété vraie au rang $l-1$, c'est-à dire que $T(f_{n,0})^{l-1}$ converge vers $T(f_{0,0})^{l-1}$.

Au rang $l$, si $\psi$ est une $(k-l,k-l)$ forme test, on a 

\begin{equation*}
\begin{split}
\langle T(f_{n,0})^{l}, \psi \rangle &= \langle \omega \wedge T(f_{n,0})^{l-1}, \psi \rangle + \langle dd^c g_n \wedge T(f_{n,0})^{l-1} , \psi \rangle\\
&=\langle T(f_{n,0})^{l-1}, \psi \wedge \omega \rangle + \langle g_n , T(f_{n,0})^{l-1} \wedge dd^c \psi \rangle.\\
\end{split}
\end{equation*}

Le premier terme converge vers $\langle T(f_{0,0})^{l-1}, \psi \wedge \omega \rangle$ par hypothèse de récurrence. Pour le second, on a 

\begin{equation*}
\begin{split}
& | \langle g_n , T(f_{n,0})^{l-1} \wedge dd^c \psi \rangle - \langle g_0 , T(f_{0,0})^{l-1} \wedge dd^c \psi \rangle |  \\
&= |\langle g_n -g_0, T(f_{n,0})^{l-1} \wedge dd^c \psi \rangle + \langle g_0 , T(f_{n,0})^{l-1} \wedge dd^c \psi - T(f_{0,0})^{l-1} \wedge dd^c \psi  \rangle |  \\
&\leq  |\langle g_n -g_0, T(f_{n,0})^{l-1} \wedge dd^c \psi \rangle| + | \langle g_0 , T(f_{n,0})^{l-1} \wedge dd^c \psi - T(f_{0,0})^{l-1} \wedge dd^c \psi  \rangle |=a_n + b_n.
\end{split}
\end{equation*}

Le terme $a_n$ tend vers $0$. En effet, on a $0 \leq a_n \leq C(\psi) \| g_n -g_0 \|_{\infty} \| T(f_{n,0})^{l-1} \|=C(\psi) \| g_n -g_0 \|_{\infty}$ en utilisant une majoration de $dd^c \psi$ comme dans la fin de la preuve de la proposition \ref{prop9} et on a démontré que $(g_n)_{n \geq 1}$ converge uniformément vers $g_0$ quand $n$ tend vers l'infini.

Le terme $b_n$ tend vers $0$ aussi car $T(f_{n,0})^{l-1} \wedge dd^c \psi$ converge vers $T(f_{0,0})^{l-1} \wedge dd^c \psi $ par hypothèse de récurrence et $g_0$ est continue. 

Finalement, $\langle T(f_{n,0})^{l}, \psi \rangle $ converge vers

$$\langle T(f_{0,0})^{l-1}, \psi \wedge \omega \rangle + \langle g_0 , T(f_{0,0})^{l-1} \wedge dd^c \psi \rangle |= \langle T(f_{0,0})^{l}, \psi \rangle.$$

Cela termine la récurrence. 

\bigskip

Donnons une application de ce théorème de convergence (voir aussi les propriétés de continuité dans \cite{FW}).

Soit $K$ un compact de $\Pp^N(\Cc)$ disjoint de $\mathcal{M}$ et $F$ un endomorphisme holomorphe de $\Pp^N(\Cc)$. Si les $f_{n,i}$ sont dans $K$ (pour $i,n \in \Nn$), on a $ dist(F^i(f_{n,0}), \mathcal{M}) \geq dist(K, \mathcal{M}) \geq \epsilon_0 > 0$ pour tout $n$ et $i$ et l'hypothèse (H) est vérifiée. Ce cas se produit par exemple si $f_{0,0} \in \mathcal{H}_d$ est un point fixe attractif de $F$. En effet, soit $K= \overline{B(f_{0,0}, \epsilon)}$ avec $\epsilon>0$ suffisamment petit pour que $K$ soit dans le bassin d'attraction de $f_{0,0}$, $F(K) \subset K$ et $K$ soit disjoint de $\mathcal{M}$. On part de $f \in K$ et on note $f_{n,0}=F^n(f)$ pour $n \geq 1$. Alors $(f_{n,0})$ converge vers $f_{0,0}$ et tous les $f_{n,i}=F^{i}(f_{n,0})$ sont dans $K$. On en déduit que $T(f_{n,0})^l$ converge vers $T(f_{0,0})^l$.

\section{\bf{Mélange aléatoire et théorème de récurrence}}

Dans ce paragraphe, nous montrons un théorème de mélange aléatoire ensuite nous l'utilisons pour prouver un résultat de récurrence sur les mesures $\mu(f_i)$ construites précédemment.

\subsection{\bf{Mélange aléatoire}}

\begin{theoreme}{\label{melange}}

On considère une suite $(f_n)_{n \in \Nn}$ d'endomorphismes holomorphes de degrés $d \geq 2$ et une suite de probabilités $(\mu(f_n))_{n \in \Nn}$ telle que pour tout $n \in \Nn$ on ait $f_n^*(\mu(f_{n+1}))=d^k \mu(f_{n})$ et $\mu(f_n)=(\omega + dd^c g_n)^k$ avec $g_n$ des fonctions continues.

Alors pour $\varphi \in L^{\infty}(\Pp^k)$ et $\psi \in DSH(\Pp^k)$, on a

$$| \langle \mu(f_0), (f_{n-1} \circ \cdots \circ f_0)^* \varphi \psi \rangle - \langle \mu(f_n) , \varphi \rangle \langle \mu(f_0), \psi \rangle | \leq C d^{-n} (1+ \| g_n \|_{\infty}  )^2 \|  \varphi \|_{\infty} \| \psi \|_{\mbox{DSH}}.$$

Ici $C$ est une constante qui ne dépend que de $\Pp^k$.

\end{theoreme}

Lorsque tous les $f_n$ sont égaux, ce résultat est le théorème de mélange classique (voir \cite{FS1} et \cite{DNS}). 

La démonstration du théorème va suivre la méthode utilisée par Dinh, Nguyen et Sibony dans \cite{DNS} pour démontrer un théorème de mélange pour les endomorphismes holomorphes de $\Pp^k(\Cc)$. Commençons par démontrer quelques lemmes qui nous serviront ensuite.

\begin{lemme}{\label{lemme1}}

Pour $n \in \Nn$, notons $\mbox{DSH}_{0, \mu(f_n)}= \{ \varphi \mbox{ DSH, } \langle \varphi, \mu(f_n) \rangle=0 \}$. Sous les mêmes hypothèses que dans le théorème précédent, on a pour tout $n \in \Nn$

$$\mbox{DSH}_{0, \mu(f_n)} \xrightarrow{(f_n)_*} \mbox{DSH}_{0, \mu(f_{n+1})}.$$

\end{lemme}

\begin{proof}

Soit $\varphi \in \mbox{DSH}_{0, \mu(f_n)}$. On a

$$ \langle (f_n)_* \varphi , \mu(f_{n+1}) \rangle= \langle \varphi , f_n^* \mu(f_{n+1}) \rangle= d^k \langle \varphi , \mu(f_n) \rangle=0.$$

\end{proof}

\begin{lemme}{\label{lemme2}}

Pour tout $\varphi \in \mbox{DSH}_{0, \mu(f_n)}$ et tout $n \in \Nn$, on a

$$ \| (f_n)_* \varphi \|_{\mbox{DSH}}^{\mu(f_{n+1})} \leq d^{k-1} \| \varphi \|_{\mbox{DSH}}^{\mu(f_n)}.$$

\end{lemme}

\begin{proof}

On a 

$$ \| (f_n)_* \varphi \|_{\mbox{DSH}}^{\mu(f_{n+1})}= | \langle (f_n)_* \varphi, \mu(f_{n+1}) \rangle | + \inf \| S^{\pm} \|$$

où $S^{\pm}$ sont des $(1,1)$ courants positifs fermés qui vérifient $dd^c (f_n)_* \varphi=S^+ - S^-$.

Mais si $dd^c \varphi = R^+ - R^-$, on a $dd^c (f_n)_* \varphi= (f_n)_* R^+ - (f_n)_* R^-$. En particulier,

$$\| (f_n)_* \varphi \|_{\mbox{DSH}}^{\mu(f_{n+1})} \leq | \langle (f_n)_* \varphi, \mu(f_{n+1}) \rangle | + \inf \| (f_n)_* R^{\pm} \|$$

avec $R^{\pm}$ des $(1,1)$ courants positifs fermés tels que $dd^c \varphi = R^+ - R^-$.

Comme les $f_n$ sont des endomorphismes holomorphes de degré $d$ de $\Pp^k(\Cc)$, on a $\| (f_n)_* R^{\pm} \| = d^{k-1} \| R^{\pm} \|$ et le lemme provient du fait que $| \langle (f_n)_* \varphi, \mu(f_{n+1}) \rangle |=0$ par le lemme \ref{lemme1}.

\end{proof}

Démontrons maintenant le théorème.

On considère $\varphi \in L^{\infty}(\Pp^k)$ et $\psi \in DSH(\Pp^k)$ et on pose $\psi_0= \psi - \langle \mu(f_0), \psi \rangle$. On a donc $\langle  \mu(f_0) , \psi_0 \rangle =0$.

On va calculer $| \langle \mu(f_0), (f_{n-1} \circ  \cdots \circ f_0)^* \varphi \psi_0 \rangle |$ de deux façons différentes.

Tout d'abord,

\begin{equation*}
\begin{split}
| \langle \mu(f_0), (f_{n-1} \circ  \cdots \circ f_0)^* \varphi \psi_0 \rangle | &= | \langle \mu(f_0), (f_{n-1} \circ  \cdots \circ f_0)^* \varphi \psi \rangle  -  \langle \mu(f_0), (f_{n-1} \circ  \cdots \circ f_0)^* \varphi \langle \mu(f_0) , \psi \rangle \rangle | \\
&=| \langle \mu(f_0), (f_{n-1} \circ  \cdots \circ f_0)^* \varphi \psi \rangle -  \langle (f_{n-1} \circ  \cdots \circ f_0)_* \mu(f_0), \varphi  \rangle \langle  \mu(f_0) , \psi \rangle| \\
&=| \langle \mu(f_0), (f_{n-1} \circ \cdots \circ f_0)^* \varphi \psi \rangle - \langle \mu(f_n) , \varphi \rangle \langle \mu(f_0), \psi \rangle | \\
\end{split}
\end{equation*}
car l'hypothèse du théorème implique que $(f_n)_* \mu(f_n)= \mu(f_{n+1})$. On a obtenu ainsi la quantité que l'on souhaite majorer.

Par ailleurs, si on note $\Lambda_i:= \frac{(f_i)_*}{d^k}$, on a

\begin{equation*}
\begin{split}
| \langle \mu(f_0), (f_{n-1} \circ  \cdots \circ f_0)^* \varphi \psi_0 \rangle | &= \left| \left\langle \frac{(f_{n-1} \circ  \cdots \circ f_0)^*\mu(f_n)}{d^{kn}}, (f_{n-1} \circ  \cdots \circ f_0)^* \varphi \psi_0 \right\rangle \right| \\
&= | \langle \mu(f_n), \varphi \Lambda_{n-1} \cdots \Lambda_0 \psi_0 \rangle | \leq  \langle \mu(f_n), | \varphi | | \Lambda_{n-1} \cdots \Lambda_0 \psi_0 | \rangle  \\
& \leq \| \varphi \|_{\infty}  \langle \mu(f_n), | \Lambda_{n-1} \cdots \Lambda_0 \psi_0 | \rangle  \leq \| \varphi \|_{\infty} \|  |\Lambda_{n-1} \cdots \Lambda_0 \psi_0 | \|_{\mbox{DSH}}^{\mu(f_n)}\\
& \leq C' \| \varphi \|_{\infty} (1 + \|g_n\|_{\infty} ) \| |\Lambda_{n-1} \cdots \Lambda_0 \psi_0 | \|_{\mbox{DSH}}
\end{split}
\end{equation*}

la dernière inégalité venant de la proposition \ref{estimation1}. 

Mais, il existe une contante $C$ telle que pour toute fonction DSH $\psi$ on ait $\| | \psi | \|_{\mbox{DSH}} \leq C \| \psi \|_{ \mbox{DSH}}$ (voir \cite{DNS}), d'où en utilisant aussi la proposition \ref{estimation2}

\begin{equation*}
\begin{split}
| \langle \mu(f_0), (f_{n-1} \circ  \cdots \circ f_0)^* \varphi \psi_0 \rangle | \leq CC' \| \varphi \|_{\infty} (1 + \|g_n\|_{\infty} ) \| \Lambda_{n-1} \cdots \Lambda_0 \psi_0  \|_{\mbox{DSH}}\\
\leq C'' \| \varphi \|_{\infty} (1 + \|g_n\|_{\infty} )^2 \| \Lambda_{n-1} \cdots \Lambda_0 \psi_0  \|_{\mbox{DSH}}^{\mu(f_n)}.\\
\end{split}
\end{equation*}

Enfin, comme $\psi_0 \in DSH_{0, \mu(f_0)}$, en utilisant le lemme \ref{lemme1}, on a $\Lambda_{n-1} \cdots \Lambda_0 \psi_0 \in DSH_{0, \mu(f_n)}$ et le lemme \ref{lemme2} implique alors

\begin{equation*}
\begin{split}
| \langle \mu(f_0), (f_{n-1} \circ  \cdots \circ f_0)^* \varphi \psi_0 \rangle | & \leq C'' d^{-n} \| \varphi \|_{\infty} (1 + \|g_n\|_{\infty} )^2 \|  \psi_0 \|_{\mbox{DSH}}^{\mu(f_0)}\\
& \leq C'' d^{-n} \| \varphi \|_{\infty} (1 + \|g_n\|_{\infty} )^2 \|  \psi \|_{\mbox{DSH}}.\\
\end{split}
\end{equation*}

Cela démontre le théorème.

\subsection{\bf{Applications, propriétés de récurrence}}

Pour pouvoir appliquer le théorème précédent, il faut contrôler $(1 + \|g_n\|_{\infty} )^2$. On va donner ici une situation générale où c'est le cas et ensuite nous donnerons un exemple où l'on peut l'appliquer.

Considérons toujours une application mesurable $F$ de $\Pp^N(\Cc)$ dans $\Pp^N(\Cc)$ et $\Lambda$ une mesure ergodique et invariante par $F$. On suppose que $\int \log dist(f , \mathcal{M}) d \Lambda(f) > - \infty$ et on note encore $A$ l'ensemble des bons points pour le théorème de Birkhoff associés à la mesure $\Lambda$ et la fonction intégrable $\log dist(f , \mathcal{M})$.

Dans le paragraphe \ref{courant}, nous avons montré que pour $f_0 \in A$, on a $f_n:=F^n(f_0) \in A$ pour tout $n$ et on peut donc construire les mesures de Green associées à chaque $f_n$. Comme $F^{i}(f_n)=f_{n+i}$, on peut écrire ces mesures sous la forme suivante: 

$$\mu(f_n)=\left( \omega + dd^c  \left( \sum_{i=0}^{+ \infty} \frac{u_{n+i} \circ f_{n+i-1} \circ \cdots \circ f_{n}}{d^{i}} \right) \right)^k=:(\omega + dd^c g_n)^k.$$

Ici les $g_n$ sont des fonctions continues et on a vu que $f_n^*(\mu(f_{n+1}))=d^k \mu(f_{n})$. 

On peut donc appliquer le théorème de mélange à la suite $(f_n)_{n \in \Nn}$ et on a pour $\varphi \in L^{\infty}(\Pp^k)$ et $\psi \in DSH(\Pp^k)$:

$$| \langle \mu(f_0), (f_{n-1} \circ \cdots \circ f_0)^* \varphi \psi \rangle - \langle \mu(f_n) , \varphi \rangle \langle \mu(f_0), \psi \rangle | \leq C d^{-n} (1+ \| g_n \|_{\infty}  )^2 \|  \varphi \|_{\infty} \| \psi \|_{\mbox{DSH}}.$$

Maintenant, on peut contrôler $(1 + \|g_n\|_{\infty} )^2$ car on a

\begin{lemme}

Soit $\epsilon >0$. Il existe $n_0$ tel que pour $n \geq n_0$ on ait $\| g_n \|_{\infty} \leq e^{\epsilon n}$.

\end{lemme}

\begin{proof}

Comme $f_0 \in A$, on a pour $\alpha >0$ (que l'on prend tel que $\alpha p < \min(\epsilon, \log d)$) l'existence par le lemme \ref{lemme} et la proposition \ref{prop2} d'un $n_0$ avec $\| u_i \|_{\infty} \leq C e^{i \alpha p}$ pour $i \geq n_0$.

On a donc pour $n \geq n_0$,

$$\| g_n \|_{\infty} \leq \sum_{i=0}^{+ \infty} \frac{ \| u_{n+i} \|_{\infty}}{d^{i}} \leq C \sum_{i=0}^{\infty} \frac{e^{(i+n) \alpha p}}{d^{i}}=C \left( \sum_{i=0}^{\infty} \frac{e^{i\alpha p}}{d^{i}} \right) e^{n \alpha p}.$$

Comme $\alpha p < \log d$, $C \left( \sum_{i=0}^{\infty} \frac{e^{i\alpha p}}{d^{i}} \right)$ est une constante et on a alors,

$$\| g_n \|_{\infty} \leq e^{\epsilon n}$$

si $n$ est assez grand. Cela démontre la proposition.

\end{proof}

On obtient alors

\begin{corollaire}

On se place sous les hypothèses précédentes.

Soient $\varphi$ une fonction continue et $\psi$ une fonction DSH.

S'il existe une sous-suite avec $\mu(f_{\alpha(n)})$ qui converge vers $\mu(f_0)$, on a 

$$\langle \mu(f_0), (f_{\alpha(n)-1} \circ \cdots \circ f_0)^* \varphi \psi \rangle \rightarrow \langle \mu(f_0) , \varphi \rangle \langle \mu(f_0), \psi \rangle.$$

\end{corollaire}

\begin{proof}

Soit $\epsilon >0$ avec $\epsilon < \frac{\log d}{2}$. Par le lemme précédent, on a l'existence de $n_0$ avec $\| g_n \|_{\infty} \leq e^{\epsilon n}$ pour $n \geq n_0$. Maintenant,

$$|\langle \mu(f_0), (f_{\alpha(n)-1} \circ \cdots \circ f_0)^* \varphi \psi \rangle - \langle \mu(f_0) , \varphi \rangle \langle \mu(f_0), \psi \rangle| \leq a_n + b_n$$

avec 

$$a_n= |\langle \mu(f_0), (f_{\alpha(n)-1} \circ \cdots \circ f_0)^* \varphi \psi \rangle - \langle \mu(f_{\alpha(n)}) , \varphi \rangle \langle \mu(f_0), \psi \rangle|$$

et 

$$b_n= | \langle \mu(f_{\alpha(n)}) , \varphi \rangle \langle \mu(f_0), \psi \rangle - \langle \mu(f_0) , \varphi \rangle \langle \mu(f_0), \psi \rangle| .$$

La suite $(b_n)$ tend vers $0$ car $\mu(f_{\alpha(n)})$ converge vers $\mu(f_0)$ et $\varphi$ est continue.

Pour $(a_n)$, on a 

$$a_n \leq Cd^{-\alpha(n)}(1+ e^{\epsilon \alpha(n)})^2  \| \varphi \|_{\infty} \| \psi \|_{\mbox{DSH}}$$

ce qui implique que $(a_n)$ tend vers $0$ aussi.

\end{proof}

Donnons maintenant une situation où l'on peut appliquer ce qui précède.

Considérons $F$ une application continue de $\Pp^N(\Cc)$ dans  $\Pp^N(\Cc)$ et $\Lambda$ une mesure ergodique et invariante par $F$. On suppose de plus que le support de $\Lambda$ est disjoint de $\mathcal{M}$.

L'ensemble des points récurrents (i.e. les $f_0$ tels qu'il existe une sous-suite avec $F^{\alpha(n)}(f_0)$ qui converge vers $f_0$) est de mesure pleine pour $\Lambda$. On notera $R$ l'intersection de cet ensemble avec $A$ et le support de $\Lambda$. Si $f_0 \in R$, on pose $f_{0,0}=f_0$ et $f_{n,0}=F^{\alpha(n)}(f_0)$ où la sous-suite est prise de sorte que $(f_{n,0})$ converge vers $f_0$.

Comme les $F^{i}(f_{n,0})$ sont dans le support de $\Lambda$ qui est un compact disjoint de $\mathcal{M}$, l'hypothèse (H) du paragraphe \ref{convergence} est vérifiée. En particulier, on a $\mu(f_{n,0})$ qui converge vers $\mu(f_{0,0})$, c'est-à-dire $\mu(f_{\alpha(n)}) \rightarrow \mu(f_0)$. On peut donc appliquer le corollaire précédent. 

Nous verrons dans le deuxième article des théorèmes de récurrence beaucoup plus forts.

\bigskip

\bigskip\noindent
Henry De Thélin, Laboratoire Analyse, Géométrie et Applications, UMR 7539,
Institut Galilée, Université Paris 13, 99 Avenue J.-B. Clément, 93430 Villetaneuse, France.\\
 {\tt dethelin@math.univ-paris13.fr}

\end{document}